\documentclass{article}
\usepackage{amsmath,amsthm,amssymb,amsfonts,mathtools,enumerate,enumitem}

\theoremstyle{theorem}
\newtheorem{theorem}{Theorem}

\newtheorem{proposition}{Proposition}

\theoremstyle{definition}

\newtheorem{assumption}{\textbf{A}\ignorespaces}

\theoremstyle{remark}

\usepackage{hyperref}
\hypersetup{
    colorlinks,
    citecolor=black,
    filecolor=red,
    linkcolor=black,
    urlcolor=red
}
\usepackage{indentfirst}
\numberwithin{equation}{section}

\usepackage{mathtools}
\usepackage{authblk}
\usepackage{mathrsfs} 
\usepackage{setspace}
\usepackage{comment}
\usepackage{changepage}
\usepackage[compact]{titlesec}

\allowdisplaybreaks

\usepackage[round]{natbib}   
\AtBeginDocument{}

\def\assumptionautorefname{}

\makeatletter
\renewcommand{\assumptionautorefname}{A\@gobble}
\makeatother

\usepackage{enumitem,amssymb}
\setlist{nolistsep}

\setenumerate[1]{label=\Roman*.}
\setenumerate[2]{label=\Alph*.}
\setenumerate[3]{label=\roman*.}
\setenumerate[4]{label=\alph*.}

\begin{document}

\title{Solutions of Poisson’s Equation for Stochastically Monotone Markov Chains}
\author[1,2]{Peter W. Glynn}
\author[2]{Alex Infanger}
\affil[1]{Department of Management Science \& Engineering, Stanford University.}
\affil[2]{Institute for Computational \& Mathematical Engineering, Stanford University.}
\date{}
\maketitle

\begin{abstract} \noindent 
Stochastically monotone Markov chains arise in many applied domains, especially in the setting of queues and storage systems. Poisson’s equation is a key tool for analyzing additive functionals of such models, such as cumulative sums of waiting times or sums of rewards. In this paper, we show that when the reward function for such a Markov chain is monotone, the solution of Poisson’s equation is monotone. This implies that the value function associated with infinite horizon average reward is monotone in the state when the reward is monotone.
\end{abstract}
\section{Introduction}
Many stochastic models arising in queueing and storage applications can be formulated as Markov chains that are stochastically monotone, in the sense that if the process is initialized with more ``work'' in the system, then the system remains more congested. To be more precise, suppose that $X=(X_n:n\geq 0)$ takes values in a state space $S\subseteq \mathbb{R}$. For $x\in S$, let $P_x(\cdot|X_0=x)$ be the probability on the path-space of $X$ conditional on $X_0=x$ and let $E_x(\cdot)=E(\cdot|X_0=x)$ be its associated expectation. We say that $X$ is \emph{stochastically monotone} if for each $y\in S$, $P_x(X_1>y)$ is non-decreasing as a function of $x\in S$. Many birth-death chains on $\mathbb{Z}_+$ are stochastically monotone, as is the waiting time sequence $W=(W_n:n\geq 0)$ satisfying the recursion
\begin{align*}
W_{n+1} = [W_n+Z_{n+1}]^+
\end{align*}
for $n\geq 0$, where $Z_1,Z_2,...,$ are independent and identically distributed (iid) random variables (rv's) and $W_0\in S=\mathbb{R}_+$. Stochastically monotone Markov chains are ubiquitous within single station queueing environments; see \citet{stoyanComparisonMethodsQueues1983}, \citet{lundGeometricConvergenceRates1996}, and \citet*{lundComputableExponentialConvergence1996} for additional discussion and examples.

Consider such a Markov chain $X$, having a stationary distribution $\pi=(\pi(dx):x\in S)$. Assume $r:S\rightarrow\mathbb{R}$ is a ``reward'' function which is non-decreasing, so that $r(x)\leq r(y)$ for $x,y\in S$ with $x\leq y$. This covers the great majority of performance measures used within the queueing setting. Let $P=(P(x,dy):x,y\in S)$ be the one-step transition kernel (or matrix, when $S$ is discrete), and let
\begin{align*}
\pi r =\int_{S}^{}r(y)\pi(dy)
\end{align*}
be the steady-state expectation of $r(X_n)$. A fundamental object of interest in the analysis of Markov chains is \emph{Poisson's equation}, in which we wish to find the function $g$ for which
\begin{align}
(P-I)g=-r_c,\label{eq11}
\end{align}
where $r_c$ is the centered reward function given by $r_c(x)=r(x)-\pi r$.

The solution $g$ to \eqref{eq11} plays a key role in the analysis of the cumulative reward functional $S_n(r)=\sum_{j=0}^{n-1}r(X_j)$. In particular, the sequence of rv's defined by 
\begin{align}
g(X_n) + \sum_{j=0}^{n-1}r(X_j)-n\cdot \pi r\label{eq12}
\end{align}
is then a martingale (in the presence of appropriate integrability). The martingale structure then implies that
\begin{align*}
E_x S_n(r) = n\pi r + g(x)-E_xg(X_n).
\end{align*}
When $X$ is aperiodic and $g$ is $\pi$-integrable, this leads to the approximation, valid for large time horizons $n$, given by 
\begin{align*}
E_xS_n(r)\approx n\pi r + g(x)-\pi g.
\end{align*}
The martingale representation \eqref{eq12} also allows one to apply martingale theory to establish central limit theorems (CLT's) and laws of the iterated logarithm (LIL's) for $S_n(r)$; see, for example, \citet{maigretTheoremeLimiteCentrale1978} and \citet{glynnSolvingPoissonEquation2022}.

We note also that the value function $v=(v(x):x\in S)$ for a Markov decision process (MDP) associated with maximizing the long-run average reward $n^{-1}S_n(r)$ satisfies \eqref{eq11}, where $P$ is the transition kernel (or matrix) associated with the dynamics of $X$ under the optimal policy.

In view of the importance of Poisson's equation, this paper establishes the following fundamental result. In particular, when $X$ is stochastically monotone and $r$ is non-decreasing, the solution $g$ to Poisson's equation must necessarily also be non-decreasing. This monotonicity property implies, for example, that when constructing approximating numerical schemes for computing MDP value functions or value functions for uncontrolled Markov chains, the approximation should ideally be monotone when $X$ is suitably monotone. It turns out that the monotonicity of $g$ also plays a key role in studying truncation schemes for numerically computing stationary distributions for stochastically monotone chains; see \citet*{infangerConvergenceGeneralTruncationAugmentation2022}.

Our proof uses coupling arguments that may be of separate interest. Section \ref{sec2} contains statements of our main results, as well as our proofs.

\section{The Monotonicity Results}\label{sec2}
We discuss here our monotonicity results, starting with a simple proof that covers most examples, and then developing theory that covers richer classes of models. Our theory hinges on the following coupling of the Markov chain starting from $x\in S$ with the chain starting from $y\in S$. In particular, let
\begin{align*}
F(x,y) = P_x(X_1\leq y)
\end{align*}
for $x,y\in S$, and set
\begin{align*}
F^{-1}(x,y)=\inf\{z:F(x,z)\geq y\}.
\end{align*}
For $x\in S$, let $X_0(x)=x$ and put
\begin{align*}
X_{n+1}(x)=F^{-1}(X_n(x),U_{n+1})
\end{align*}
for $n\geq 0$, where $U_1,U_2,...$ is an iid sequence of rv's uniformly distributed on $[0,1]$ under the probability $P$ (having associated expectation $E$). The key observation is that for each $x\in S$,
\begin{align*}
P((X_j(x):j\geq 0)\in \cdot) =P_x((X_j:j\geq 0)\in \cdot),
\end{align*}
so that $X(x)=(X_j(x):j\geq 0)$ has the same distribution as $X$ under $P_x$. This coupling goes back to \citet*{kamaeStochasticInequalitiesPartially1977}.

Because $X$ is stochastically monotone, $F^{-1}(\cdot,y)$ is non-decreasing for each $y\in S$. Consequently, $F^{-1}(\cdot,U_{n+1})$ is non-decreasing, and it follows by induction that for each $n\geq 0$, $X_n(x)$ is a non-decreasing function of $x\in S$. Hence, if $r$ is a non-decreasing function $r_c(X_n(x))$ is non-decreasing in $x\in S$. Thus, if $r_c(X_n(x))$ is integrable, it is evident that
\begin{align*}
\sum_{j=0}^{n-1}E_xr_c(X_j)
\end{align*}
is non-decreasing in $x$. As a result, if
\begin{align}
\sum_{j=0}^{\infty}|E_xr_c(X_j)|<\infty\label{eq21}
\end{align}
for $x\in S$, then
\begin{align}
g(x)\overset{\Delta}{=}\sum_{j=0}^{\infty}E_xr_c(X_j)\label{eq22}
\end{align}
must be non-decreasing in $x\in S$. But if \eqref{eq21} holds, then it is easy to see that $g$ as defined by \eqref{eq22} is a solution of Poisson's equation \eqref{eq11}. We summarize our discussion thus far with our first result.
\begin{proposition} Suppose that $X$ is stochastically monotone and that $r$ is non-decreasing. If \eqref{eq21} holds for each $x\in S$, then $g$ as defined by \eqref{eq22} is a non-decreasing solution of Poisson's equation \eqref{eq11}.
\end{proposition}
As noted in \citet{glynnSolvingPoissonEquation2022}, there exist Markov chains $X$ for which \eqref{eq11} is solvable, and yet the representation \eqref{eq22} for the solution fails to be valid. Consequently, we wish to develop alternative proofs that are more general. \citet{glynnSolvingPoissonEquation2022} show that when $r$ is $\pi$-integrable, $X$ is irreducible, and $S\subseteq \mathbb{Z}$, then
\begin{align}
g_z(x) = E_x\sum_{j=0}^{\tau-1}r_c(X_j)\label{eq23}
\end{align}
always solves \eqref{eq11}, where $\tau=\inf\{n\geq 1:X_n=z\}$ is the first return time to any (fixed) state $z\in S$. Furthermore, it is shown there that \eqref{eq12} is $P_x$-integrable and that
\begin{align*}
M_n\overset{\Delta}{=} g_z(X_n) + \sum_{j=0}^{n-1}r_c(X_j)
\end{align*}
is a $P_x$-martingale adapted to $(\mathcal{F}_n:n\geq 0)$, where $\mathcal{F}_n=\sigma(X_j:j\leq n)$.

We now use our coupling to prove the following result.

\begin{theorem}\label{thm1} Suppose that $X$ is an irreducible positive recurrent stochastically monotone Markov chain taking values in $S=\mathbb{Z}_+$ having stationary distribution $\pi$. If $r:S\rightarrow\mathbb{R}$ is a $\pi$-integrable non-decreasing function, then
\begin{align}
g_0(x) = E_x\sum_{j=0}^{\tau-1}r_c(X_j)\tag{\ref{eq23}a}\label{eq23a}
\end{align}
(with $\tau=\inf\{n\geq 1:X_n=0\}$) is a non-decreasing solution of Poisson's equation \eqref{eq11} and $(M_n:n\geq 0)$ is a $P_x$-martingale for each $x\in S$.
\end{theorem}

\begin{proof} Fix $z=0\leq x<y$. Expressed in terms of our coupling, the martingale structure of $(M_n:n\geq 0)$ implies that
\begin{align*}
M_n(x)\overset{\Delta}{=} g_0(X_n(x)) + \sum_{j=0}^{n-1}r_c(X_j(x))
\end{align*}
and
\begin{align*}
M_n(y) \overset{\Delta}{=}g_0(X_n(y)) + \sum_{j=0}^{n-1}r_c(X_j(y))
\end{align*}
are both martingales adapted to $(\mathcal{G}_n:n\geq 0)$, where $\mathcal{G}_n=\sigma(X_j(x), X_j(y):0\leq j\leq n)$. For $w\in S$, let $\tau(w)=\inf\{n\geq 0: X_n(w)=0\}$. We now wish to apply optional sampling at time $\tau(y)$ to $(M_n(x):n\geq 0)$ and $(M_n(y):n\geq 0)$. Assuming temporarily that optional sampling can be applied, we find that
\begin{align}
g_0(x) = Eg_0(X_{\tau(y)}(x)) + E\sum_{j=0}^{\tau(y)-1}r_c(X_j(x))\label{eq24}
\end{align}
and
\begin{align}
g_0(y) = Eg_0(X_{\tau(y)}(y)) + E\sum_{j=0}^{\tau(y)-1}r_c(X_j(y)).\label{eq25}
\end{align}
But under our coupling, $X_{\tau(y)}(x)\leq X_{\tau(y)}(y)=0$. Since $S=\mathbb{Z}_+$, $X_{\tau(y)}(x)\geq 0$, so it follows that $X_{\tau(y)}(x)=0$. So,
\begin{align*}
g_0(x)-g_0(0) = E\sum_{j=0}^{\tau(y)-1}r_c(X_j(x))
\end{align*}
and
\begin{align*}
g_0(y)-g_0(0) = E\sum_{j=0}^{\tau(y)-1}r_c(X_j(y)).
\end{align*}
The monotonicity of $r_c$ implies that
\begin{align}
\sum_{j=0}^{\tau(y)-1}r_c(X_j(x)) \leq \sum_{j=0}^{\tau(y)-1}r_c(X_j(y)),\label{eq26}
\end{align}
thereby implying that $g_0(x)\leq g_0(y)$.

We now need to establish the validity of optional sampling. Note that $\tau(y)\land n \overset{\Delta}{=}\min(n,\tau(y))$ is a bounded stopping time, so it follows from the optional sampling theorem (see, for example, Theorem 6.2.2 in \citet{rossStochasticProcesses1996}) that
\begin{align}
EM_{\tau(y)\land n}(x) = EM_0(x) = g_0(x).\label{eq27}
\end{align}
The validity of \eqref{eq24} follows from \eqref{eq27}, once we prove that $(M_{\tau(y)\land n}:n\geq 0)$ is a uniformly integrable sequence. Note that
\begin{align}
|M_{\tau(y)\land n}(x)|\leq |g_0(X_{\tau(y)\land n}(x))| + \sum_{j=0}^{\tau(y)-1}|r_c(X_j(x))|.\label{eq28}
\end{align}
Because $r_c$ is non-decreasing, $r_c$ can be expressed as $r_c(x)=\tilde r_c(x)+r_c(0)$, where $\tilde r_c$ is non-decreasing and non-negative. Hence,
\begin{align*}
\sum_{j=0}^{\tau(y)-1}|r_c(X_j(x))|&\leq \sum_{j=0}^{\tau(y)-1}\tilde r_c(X_j(x)) + |r_c(0)|\tau(y)\\
&\leq \sum_{j=0}^{\tau(y)-1}\tilde r_c(X_j(y)) + |r_c(0)| \tau(y)\\
&= \sum_{j=0}^{\tau(y)-1}|r_c(X_j(y))| + 2|r_c(0)|\tau(y).
\end{align*}
Recall that
\begin{align*}
E \sum_{j=0}^{\tau(y)-1}|r_c(X_j(y))| + 2|r_c(0)|\tau(y)&= E_y\sum_{j=0}^{\tau-1}|r_c(X_j)| + 2|r_c(0)|E_y\tau.
\end{align*}
The latter expectations are finite because $r_c$ is $\pi$-integrable and $X$ is positive recurrent; see Section 2 of \citet{glynnSolvingPoissonEquation2022} for details. Hence, the latter term on the right-hand side of \eqref{eq28} is integrable (and therefore trivially uniformly integrable).

We now argue that the other term appearing on the right-hand side of \eqref{eq28}, namely $(g_0(X_{\tau(y)\land n}):n\geq 0)$, is also uniformly integrable. Note that $|g_0(X_{\tau(y)\land n}(x))|\rightarrow |g_0(X_{\tau(y)}(x))|$ a.s. as $n\rightarrow\infty$. As argued earlier, $X_{\tau(y)}(x)=0$, so uniform integrability follows (see Theorem 4.6.3 of \citet{durrettProbabilityTheoryExamples2019}) if
\begin{align*}
 E|g_0(X_{\tau(y)\land n}(x))|\rightarrow |g_0(0)|=0
 \end{align*}
 as $n\rightarrow\infty$, which is a consequence of establishing that
 \begin{align}
 E|g_0(X_n(x))|I(\tau(y)\geq n)\rightarrow 0\label{eq29}
 \end{align}
 as $n\rightarrow\infty$.

 Let $\beta_n(x)=\inf\{j>n:X_j(x)=0\}$, and observe that
 \begin{align*}
 g_0(X_n(x))= E[\sum_{j=n}^{\beta_n(x)-1}r_c(X_j(x))|X_n(x)].
 \end{align*}
 In view of the fact that $\beta_n(x)\leq\beta_n(y)$, the left-hand side of \eqref{eq29} can be upper bounded by
 \begin{align*}
 E\sum_{j=n}^{\beta_n(x)-1}|r_c(X_j(x))|I(\tau(y)>n)&\leq E \sum_{j=n}^{\beta_n(x)-1}|r_c(X_j(y))|I(\tau(y)>n) \\
 &\leq E\sum_{j=n}^{\beta_n(y)-1}|r_c(X_j(y))|I(\tau(y)>n)\\
 &=E\sum_{j=n}^{\tau(y)-1}|r_c(X_j(y))|I(\tau(y)>n)\\
 &=E_y \sum_{j=n}^{\tau-1}|r_c(X_j)|I(\tau>n)\rightarrow 0,
 \end{align*}
 proving \eqref{eq29}. A similar (but easier) argument proves \eqref{eq25}, thereby proving the theorem.
 \end{proof}

We finish this section by developing the corresponding theory when the state space $S$ is a continuous state space. In particular, we assume that $S=\mathbb{R}_+$. Our first result concerns the case where $X$ is a positive recurrent Harris chain. Specifically, we will invoke the following assumption
\begin{assumption} \label{a1}\begin{adjustwidth}{50pt}{50pt} \hspace{-0.825cm}There exist $b,\lambda>0$, a probability $\phi$ on $\mathbb{R}_+$, and non-negative functions $v_i:\mathbb{R}_+\rightarrow\mathbb{R}_+$ $(i=1,2)$ such that
\begin{enumerate}[label=\alph*)]
\item $\sup\{E_xv_i(X_i):0\leq x\leq b\}<\infty, \ i=1,2$;
\item $\sup\{|r(x)|:0\leq x\leq b\}<\infty$;
\item $E_xv_1(X_1)\leq v_1(x)-1,\  x>b$;
\item $E_xv_2(X_1)\leq v_2(x)-|r(x)|,\  x>b$;
\item $P_x(X_1\in \cdot)\geq \lambda\phi(\cdot),\ x\in[0,b]$.
\end{enumerate}
\end{adjustwidth}
\end{assumption}
\citet{athreyaNewApproachLimit1978} and \citet{nummelinSplittingTechniqueHarris1978} noted that condition e) allows one to ``split'' the transition kernel $P$ over $[0,b]$, so that one can then write
\begin{align}
P_x(X_1\in \cdot)= \lambda \phi(\cdot) + (1-\lambda)Q(x,\cdot)\label{eq210}
\end{align}
for $x\in[0,b]$, where $Q(x,\cdot)$ is a probability on $S$ for each $x\in[0,b]$. The mixture \eqref{eq210} then allows one to interpret transitions from $x\in[0,b]$ in terms of a ``randomization''. In particular, if $X_n\in[0,b]$, then $X_{n+1}$ is distributed according to $\phi$ with probability $\lambda$ and distributed according to $Q(X_n,\cdot)$ with probability $1-\lambda$. Every time that $X$ distributes itself according to $\phi$, the Markov chain ``regenerates.''


Let $\tau$ be the first time that $X$ distributes itself according to $\phi$. We are now ready to state a theorem that provides a continuous state space analog to the discrete state space representation \eqref{eq23} for the solution of Poisson's equation.

\begin{theorem} Let $X=(X_n:n\geq 0)$ be a Markov chain satisfying \autoref{a1}. Then, $X$ possesses a unique stationary distribution $\pi$ and $r$ is $\pi$-integrable. If $r_c(x)=r(x)-\pi r$ for $x\geq 0$, then
\begin{align*}
E_x\sum_{j=0}^{\tau-1}|r_c(X_j)|<\infty
\end{align*}
for each $x\in \mathbb{R}_+$ and
\begin{align*}
\tilde g(x)=E_x\sum_{j=0}^{\tau-1}r_c(X_j)
\end{align*}
is a finite-valued solution of Poisson's equation for which
\begin{align*}
\tilde g(X_n) + \sum_{j=0}^{n-1}r_c(X_j)
\end{align*}
is a $P_x$-martingale for each $x\in \mathbb{R}_+$.
\end{theorem}
\begin{proof} Condition e) implies that $[0,b]$ is a small set, in the terminology of \citet{meynMarkovChainsStochastic2012}. In view of conditions a) and c), we may apply Theorem 11.3.4 of \citet{meynMarkovChainsStochastic2012}, thereby ensuring that $X$ has a unique stationary distribution $\pi$. Furthermore, conditions a), b), and d), together with Theorem 14.3.7 of \citet{meynMarkovChainsStochastic2012} imply that
\begin{align*}
\int_{\mathbb{R}_+}^{}\pi(dx)|r(x)|\leq \sup\{E_xv_2(X_1)-v_2(x)+|r(x)|: 0\leq x\leq b\}<\infty.
\end{align*}
Since $\pi r$ is finite-valued, we may put $r_c(x)=r(x)-\pi r$. 

In view of conditions a)-d), we may apply the Comparison Theorem (p. 344, \citet{meynMarkovChainsStochastic2012}) to conclude that
\begin{align}
E_x\sum_{j=0}^{\tau-1}|r_c(X_j)| \leq v_2(x)+ |\pi r| v_1(x)+\beta E_x\sum_{j=0}^{\tau-1}I(X_j\leq b)\label{eq211}
\end{align}
where
\begin{align*}
\beta=\sup\{|\pi r|v_1(x)+v_2(x):0\leq x\leq b\}.
\end{align*}
With the aid of the splitting idea discussed earlier, we see that
\begin{align*}
P_x(\tau>T_{k+1}|X_0,...,X_{T_k})=(1-\lambda),
\end{align*}
where $T_k$ is the time step at which $X$ visits $[0,b]$ for the $k$'th time, and hence
\begin{align*}
P(\tau>T_k) = (1-\lambda)^k.
\end{align*}
It follows that
\begin{align*}
E_x \sum_{j=0}^{\tau-1}I(X_j\leq b) = E_x \sum_{j=1}^{\infty}I(\tau>T_j) = \frac{1}{\lambda}.
\end{align*}
Consequently, \eqref{eq211} implies that
\begin{align}
E_x\sum_{j=0}^{\tau-1}|r_c(X_j)|<\infty\label{eq212}
\end{align}
for all $x\in \mathbb{R}$.

With the knowledge that $\tilde g$ is finite-valued, then by conditioning on $X_1$, we find that
\begin{align}
\tilde g(x)=r_c(x) + \int_{\mathbb R_+}^{}\tilde g(y)P_x(X_1\in dy)\label{eq213}
\end{align}
for $x>b$, whereas
\begin{align}
\tilde g(x) = r_c(x) + (1-\lambda)\int_{\mathbb{R}_+}^{}\tilde g(y)Q(x,dy)\label{eq214}
\end{align}
for $x\leq b$.

Since $X$ regenerates at time $\tau$ and has distribution $\phi$ at that time, it is well known (see, for example, \citet{smithRegenerativeStochasticProcesses1955}) that
\begin{align}
0=\pi r_c=\frac{E_\phi\sum_{j=0}^{\tau-1}r_c(X_j)}{E_\phi \tau} = \frac{\int_{\mathbb{R_+}}^{}\phi(dx)\tilde g(x)}{E_\phi \tau},\label{eq215}
\end{align}
(where $E_\phi(\cdot)$ is the expectation under which $X_0$ has distribution $\phi$) so that \eqref{eq214} becomes
\begin{align*}
\tilde g(x)&=r_c(x) + (1-\lambda)\int_{\mathbb R_+}^{}\tilde g(y) Q(x,dy)+ \lambda\int_{\mathbb R_+}^{}\tilde g(y)\phi(dy)\\
&=r_c(x) + \int_{\mathbb{R}_+}^{}\tilde g(y) P_x(X_1\in dy),
\end{align*}
proving that Poisson's equation is solved by $\tilde g$.

To prove the martingale property, we note that
\begin{align}
E_x|r_c(X_n)| = E_x |r_c(X_n)|I(\tau>n) + \sum_{j=1}^{n}P_x(\tau=j)E_\phi|r_c(X_{n-j})|\label{eq216}
\end{align}
and 
\begin{align}
E_\phi |r_c(X_n)| = E_\phi |r_c(X_n)| I(\tau>n) + \sum_{j=1}^{n}P_\phi(\tau=j)E_\phi |r_c(X_{n-j})|\label{eq217}
\end{align}
(where $P_\phi(\cdot)$ is the probability under which $X_0$ has distribution $\phi$). But the $\pi$-integrability of $r$ ensures that
\begin{align}
E_\phi \sum_{j=0}^{\tau-1}|r_c(X_j)| <\infty\label{eq218}
\end{align}
(using the same regenerative identity as in \eqref{eq215}), so that $E_\phi|r_c(X_n)|I(\tau>n)<\infty$ for $n\geq 0$. Also, \eqref{eq212} guarantees that $E_x|r_c(X_n)|I(\tau>n)<\infty$ for $n\geq 0$. It follows from an induction based on the recursions \eqref{eq216} and \eqref{eq217} that $r_c(X_n)$ is $P_x$-integrable for $n\geq 0$.

We can similarly analyze the $P_x$-integrability of $\tilde g(X_n)$. The same inductive argument, based on equations analogous to \eqref{eq216} and \eqref{eq217}, establishes the integrability provided that we show
\begin{align*}
E_x|\tilde g(X_n)|I(\tau>n)<\infty
\end{align*}
and
\begin{align*}
E_\phi|\tilde g(X_n)| I(\tau>n) <\infty
\end{align*}
for $n\geq 0$. Set $\beta_n=\inf\{j>n: X \text{ regenerates at time }j\}$. Then,
\begin{align*}
E_x|\tilde g(X_n)| I(\tau>n) &= E_x |\sum_{j=n}^{\beta_n-1}r_c(X_j)| I(\tau>n)\\
&\leq E_x \sum_{j=n}^{\beta_n-1}|r_c(X_j)|I(\tau>n)\\
&=E_x\sum_{j=n}^{\tau-1}|r_c(X_j)| I(\tau>n)<\infty,
\end{align*}
due to \eqref{eq212}. Similarly, 
\begin{align*}
E_\phi |\tilde g(X_n)| I(\tau>n) \leq E_\phi \sum_{j=n}^{\tau-1}|r_c(X_j)|I(\tau>n)<\infty,
\end{align*}
due to \eqref{eq218}. Hence,
\begin{align*}
\tilde g(X_n) + \sum_{j=0}^{n-1}r_c(X_j)
\end{align*}
is $P_x$-integrable for $n\geq 0$. The martingale property then follows easily from the fact that $\tilde g$ solves Poisson's equation.
\end{proof}

We now proceed to prove that $\tilde g$ is monotone when $X$ and $r$ are suitably monotone.

\begin{theorem} \label{thm3} Suppose that $X$ is a stochastically monotone Markov chain satisfying \autoref{a1}. If the function $r$ appearing in \autoref{a1} is non-decreasing, then $\tilde g$ is non-decreasing.
\end{theorem}
\begin{proof} To prove this result, we modify the coupling discussed earlier. In particular, for $0\leq x\leq y$, we modify the dynamics of $((X_n(x),X_n(y)):n\geq 0)$ when $X_n(y)\leq b$.

For $w\in\mathbb{R}_+$ and $v\in[0,b]$, put
\begin{align*}
G(v,w) = \frac{F(v,w)-\lambda\phi([0,w])}{1-\lambda},
\end{align*}
and note that $1-G(\cdot,w)$ is non-decreasing for each $w\geq 0$ (since this is also true of $1-F(\cdot,w)$).

When $X_n(x)\leq X_n(y)\leq b$, we distribute $X_{n+1}(y)$ according to $\phi$ with probability $\lambda$ and put $X_{n+1}(x)=X_{n+1}(y)$. Otherwise, with probability $1-\lambda$, put $X_{n+1}(x)=G^{-1}(X_n(x), U_{n+1})$ and $X_{n+1}(y)=G^{-1}(X_n(y), U_{n+1})$. This modified coupling preserves the distribution of $X$ under $P_x$ and the distribution of $(X,\tau)$ under $P_y$ while maintaining the ordering $X_n(x)\leq X_n(y)$ for $n\geq 0$ and forcing $X_\tau(x)$ to equal $X_\tau(y)$.

Because $X(x)$ may visit $[0,b]$ earlier than $X(y)$, it may regenerate and distribute itself according to $\phi$ at a time $\tau'$ earlier than $\tau$. In particular, when $X_n(x)\leq b<X_n(y)$ (so that $X(x)$ has the potential to regenerate at time $n+1$, but $X(y)$ does not), we put $X_{n+1}(x)=F^{-1}(X_n(x),U_{n+1})$ and $X_{n+1}(y)=F^{-1}(X_n(y), U_{n+1})$. A regeneration for $X(x)$ occurs at time $n+1$ with probability $w(X_n(x), X_{n+1}(x))$, where
\begin{align*}
w(x,y) = \lambda\left[\frac{d\phi}{dP_x(X_1\in \cdot)}\right](y)
\end{align*}
is the Radon-Nikodym derivative of $\lambda \phi$ with respect to $P_x(X_1\in \cdot)$ (which exists because $\phi$ must be absolutely continuous with respect to $P_x(X_1\in \cdot)$). Then, the distribution of $(X,\tau)$ under $P_x$ matches the distribution of $(X(x), \tau')$. 

The rest of the argument follows the proof used to verify Theorem \ref{thm1}. In particular, we consider the martingales
\begin{align}
\tilde g(X_n(x)) + \sum_{j=0}^{n-1}r_c(X_j(x))\label{eq219}
\end{align}
and
\begin{align}
\tilde g(X_n(y)) + \sum_{j=0}^{n-1}r_c(X_j(y)),\label{eq220}
\end{align}
analogously to the martingales $(M_n(x):n\geq 0)$ and $(M_n(y):n\geq 0)$ used in Theorem \ref{thm1}. The key is then to prove that optional sampling can be applied to \eqref{eq219} and \eqref{eq220} at time $\tau$. The associated uniform integrability argument is essentially identical to that used in Theorem \ref{thm1}, and is therefore omitted, proving the result.
\end{proof}

In some applications, it is important to know that the solution $\tilde g$ to Poisson's equation is continuous on $\mathbb R_+$.

\begin{proposition} Suppose that $X_n$ is a stochastically monotone Markov chain and that $F^{-1}(\cdot,y)$ is continuous on $\mathbb R_+$ for each $y\geq 0$. If the function $r$ appearing in \autoref{a1} is continuous and non-decreasing, then $\tilde g$ is continuous.
\end{proposition}

\begin{proof} Suppose that $x_n\rightarrow x\geq 0$ as $n\rightarrow\infty$ with $x_n\leq y$ for all $n\geq 1$. The proof of \autoref{thm3} establishes that if $x<y$, then
\begin{align*}
\tilde g(x_n)-\tilde g(x) = E\sum_{j=0}^{\tau-1}[r_c(X_j(x_n))-r_c(X_j(x))].
\end{align*}
Since $r_c(F^{-1}(\cdot, U_n))$ is continuous, it follows that
\begin{align*}
\sum_{j=0}^{\tau-1}[r_c(X_j(x_n))-r_c(X_j(x))] \rightarrow 0\qquad \text{a.s.}
\end{align*}
as $n\rightarrow\infty$. Furthermore, since $r_c(X_j(\cdot))$ is non-decreasing,
\begin{align}
\sum_{j=0}^{\tau-1}|r_c(X_j(x_n))-r_c(X_j(x))| \leq 2\sum_{j=0}^{\tau-1}|r_c(X_j(y))| + 2 \tau|r_c(0)|.\label{eq221new}
\end{align}
The argument of \autoref{thm1} proves that \eqref{eq221new} is integrable, so that the Dominated Convergence Theorem applies. So $\tilde g(x_n)\rightarrow \tilde g(x)$ as $n\rightarrow\infty$, proving the continuity of $\tilde g$.
\end{proof}

Our final result concerns a class of Markov chains on $\mathbb{R}_+$ that need not be Harris recurrent and need not satisfy \autoref{a1}. In particular, we assume that for $0\leq x\leq y$, there exists $\rho<1$ such that
\begin{align}
E|F^{-1}(x,U_1) - F^{-1}(y,U_1)|^2\leq \rho |x-y|^2;\label{eq221}
\end{align}
such a Markov chain is said to be \emph{contractive on average}. Let $\kappa_n:\mathbb{R}_+\rightarrow\mathbb{R}_+$ be the random mapping defined by 
\begin{align*}
\kappa_n(x) = F^{-1}(x,U_n)
\end{align*}
for $n\geq 1$. Observe that
\begin{align*}
X_n(x)= (\kappa_n\circ \kappa_{n-1} \circ ...\circ\kappa_1)(x)
\end{align*}
has precisely the same distribution as
\begin{align*}
\tilde X_n(x) = (\kappa_1\circ\kappa_2\circ ... \circ \kappa_n)(x)
\end{align*}
for $n\geq 1$. If $X_0(x)=\tilde X_0(x), \tilde X_n(\cdot)$ enjoys the same monotonicity property as does $X_n(\cdot)$, so that $\tilde X_n(x)\leq \tilde X_n(y)$.

Assume that $r$ is Lipschitz, so that there exists $c<\infty$ for which
\begin{align*}
|r(x)-r(y)|\leq c^{1/2}|x-y|.
\end{align*}
Then, \eqref{eq221} implies that
\begin{align}
E|r(\tilde X_{k+1}(x))-r(\tilde X_k(x))|^2&\leq c E|\tilde X_{k+1}(x) - \tilde X_k(x)|^2\nonumber\\
&=c E|F^{-1}((\kappa_2\circ ...\circ \kappa_{k+1}(x), U_1))-F^{-1}((\kappa_2\circ...\circ\kappa_k)(x),U_1)|^2\nonumber\\
&\leq c\rho E|(\kappa_2\circ...\kappa_{k+1}(x))-(\kappa_2\circ ...\kappa_k)(x)|^2\nonumber\\
&=c \rho E|F^{-1}((\kappa_3\circ ...\circ\kappa_{k+1})(x),U_2) - F^{-1}((\kappa_3\circ...\circ \kappa_k)(x),U_2)|^2\nonumber\\
&\leq c\rho^2 E|(\kappa_3\circ ...\circ\kappa_{k+1})(x)-(\kappa_3\circ ...\circ\kappa_k)(x)|^2\nonumber\\
&\leq ...\nonumber\\
&\leq c\rho^kE|\kappa_{k+1}(x)-x|^2\nonumber\\
&=c\rho^kE|\kappa_1(x)-x|^2.\tag{\eqref{eq221}a}
\end{align}
Hence, if
\begin{align}
E_x|X_1-x|^2<\infty,\label{eq222}
\end{align}
$(r(\tilde X_k(x)):k\geq 0)$ is evidently a Cauchy sequence in the space of square integrable rv's and hence converges in mean square to a limit $r(\tilde X_\infty(x))$. \citet{diaconisIteratedRandomFunctions1999} show that \eqref{eq221} and \eqref{eq222} imply that $X$ has a unique stationary distribution $\pi$ and that $r(\tilde X_\infty(x))$ has the distribution of $r(X_0)$ under $\pi$. Hence, Cauchy-Schwarz implies that
\begin{align*}
|E_xr_c(X_n)| &= |Er(\tilde X_n(x))-Er(\tilde X_\infty(x))|\\
&\leq \sum_{j=n}^{\infty}|Er(\tilde X_j(x))-Er(\tilde X_{j+1}(x))|\\
&\leq \sum_{j=n}^{\infty}E^{1/2}|r(\tilde X_j(x))-r(\tilde X_{j+1}(x))|^2\\
&\leq c^{1/2} \sum_{j=n}^{\infty}\rho^{j/2}E^{1/2}|X_1(x)-x|^2,
\end{align*}
and consequently, 
\begin{align*}
\sum_{k=0}^{\infty}|E_xr_c(X_k)|<\infty,
\end{align*}
and a solution $g$ to Poisson's equation can then be defined by \eqref{eq22}. Since
\begin{align*}
\sum_{j=0}^{\infty} r_c(\tilde X_j(x)) \leq \sum_{j=0}^{\infty}r_c(\tilde X_j(y))
\end{align*}
for $x\leq y$, it follows that $g$ is non-decreasing. Similarly, we find that
\begin{align*}
|E_xr_c(X_n)-E_yr_c(X_n)|&\leq c^{1/2}E^{1/2}|X_n(x)-X_n(y)|^2\\
&\leq c^{1/2}\rho^{n/2}|x-y|
\end{align*}
so that
\begin{align*}
|\tilde g(x)-\tilde g(y)| \leq \frac{c^{1/2}}{1-\rho^{1/2}}|x-y|
\end{align*}
and hence $\tilde g$ is Lipschitz. We have proved our final result.

\begin{theorem} Suppose that $X$ is a stochastically monotone Markov chain satisfying \eqref{eq221} and \eqref{eq222}. If $r$ is a non-decreasing Lipschitz function, then $g$ as defined by \eqref{eq22} is a finite-valued non-decreasing Lipschitz solution to Poisson's equation \eqref{eq11}.
\end{theorem}

\bibliographystyle{newapa-and}
\bibliography{Markov-static}

\clearpage

\end{document}